\documentclass[12pt]{amsart}

\usepackage{amsmath,amsfonts,amsthm,amsopn,cite,mathrsfs}
\usepackage{epsfig,verbatim}
\usepackage{subfigure}

\newcommand{\gab}{\mathcal{G}(g,\alpha , \beta )}
\newcommand{\ltwo}{L^2(\bR )}
\newtheorem{tm}{Theorem}%[section]
\newtheorem{lemma}[tm]{Lemma}
\newtheorem{prop}[tm]{Proposition}

\newcommand{\rems}{\noindent\textsl{REMARKS:}}

\newcommand{\fif}{if and only if}

 \theoremstyle{definition}

\newcommand{\beqa}{\begin{eqnarray*}}
\newcommand{\eeqa}{\end{eqnarray*}}

\newcommand{\field}[1]{\mathbb{#1}}
\newcommand{\bR}{\field{R}}        %  real numbers
\newcommand{\bN}{\field{N}}        %  natural numbers 
\newcommand{\bZ}{\field{Z}}        %  whole numbers 
\newcommand{\bC}{\field{C}}        %  complex numbers
\newcommand{\bQ}{\field{Q}}        %  rational numbers
 \def\cF{\mathcal{F}}              % Calligraphic Letters

\def\<{\left<}
\def\>{\right>}

\def\mv1{M_v^1}

\begin{document}
\begin{abstract}
We derive new obstructions for Gabor frames. This note explains
and proves the computer generated  observations of 
Lemvig and Nielsen~\cite{LN15}.
\end{abstract}

\title{Partitions of Unity and New Obstructions for Gabor Frames}
\author{Karlheinz Gr\"ochenig}
\address{Faculty of Mathematics \\
University of Vienna \\
Nordbergstrasse 15 \\
A-1090 Vienna, Austria}
\email{karlheinz.groechenig@univie.ac.at}
\subjclass[2000]{}
\date{}
\thanks{K.\ G.\ was
  supported in part by the  project P26273 - N25  of the
Austrian Science Fund (FWF)}
\maketitle

\section{Introduction}

Given $\alpha , \beta  >0$ and $g\in \ltwo$, let $\gab = \{ e^{2\pi i
  \beta l \, \cdot } g(\cdot - \alpha k) : k,l \in \bZ \} $ be the Gabor system with
window $g$ and lattice parameters $\alpha $ and $\beta $. The basic
question is when $\gab $ generates a frame (called a \emph{Gabor frame}), i.e., when there exist
 $A,B>0$ such that 
$$
A\|f\|_2^2 \leq \sum _{k,l \in \bZ } |\langle f, e^{2\pi i
  \beta l \, \cdot } g(\cdot - \alpha k) \rangle |^2  \leq B\|f\|_2^2
\qquad \text {for all } f\in \ltwo \, .
$$
See~\cite{Gr14} for a recent survey  of Gabor frames with background
and a collection of references. 

For the characterization of Gabor frames  the Zak transform is a
standard tool. It  is defined as 
$$ Z_\alpha f (x,\xi  ) = \sum _{r\in \bZ } g(x-\alpha r) e^{2\pi i
  r\xi } \, .$$

Given $p,q \in \bN $ with $p\leq q$ and $\alpha \beta = \tfrac{p}{q}$,
we define the matrix $P(x,\xi )$ with entries 
\begin{equation}
  \label{eq:1}
P(x,\xi ) _{kl} = Z_{\alpha q} g(x+ \alpha l + \tfrac{k}{\beta} , \xi
) = Z_{\alpha q} (x+\alpha (l + \tfrac{qk}{p}),\xi ) \, .  
\end{equation}
By (quasi-)periodicity of $Z$ we may restrict the index set to  $k = 0
, \dots , p-1$, and 
$l= 0, \dots , q-1$. Thus $P(x,\xi ) $ is a $p\times q$-matrix.

Lyubarski and Nes~\cite{LN13} gave the following characterization of Gabor frames
over rational lattices.  We assume that the window is in the
Feichtinger algebra $g\in M^1(\bR )$, then the Zak transform  and the
matrix-valued function $P$ are continuous. 

\begin{lemma} \label{l1}
  Assume that $\alpha \beta = \tfrac{p}{q}\in \bQ $ with relatively
  prime $p,q $, and $g\in M^1(\bR )$, and let $P$ be the corresponding family of $p\times
  q$-matrices.
Then $\gab $ is a frame, \fif\ $P(x,\xi ) $ has   rank $p$ for all
$x,\xi \in \bR ^2$. 
\end{lemma}

 The following proposition  is a new obstruction  for Gabor frames. It
 explains rigorously  some of  the observations obtained  in
~\cite{LN15} with the help of computer algebra.

\begin{prop} \label{p2}
  Assume that $g\in M^1(\bR ) $ generates a partition of unity
  \begin{equation}
    \label{eq:2}
    \sum _{s\in \bZ } g(x-s) = 1 \qquad \text{ for all } x \in \bR  \, .
  \end{equation}
Let $m,n,r \in \bN $, $j= 1, \dots , r-1$, such that 
$
(r-1)m+1 < rn+j < rm$ and $rn+j$ and $rm$ are relatively prime. 

If $\alpha = \tfrac{1}{m}$ and $\beta = n + \tfrac{j}{r}$, then  $\gab $ is not a frame. 
\end{prop}

\begin{proof}
In this case $\alpha \beta = \tfrac{rn+j}{rm}$, $p=rn+j$,
$q=rm$, and $\alpha q = r$
  We will show  that the matrices $P(x,0)$ have rank smaller than $ p=rn+j$ for
  all $x$ and thus violate the condition of Lyubarski and Nes.
Note that $$
P(x,0)_{k,l} = Z_r g(x+\alpha l +\tfrac{k}{\beta},0)= \sum _{s\in \bZ
} g(x+\tfrac{l}{m}  +\tfrac{rk}{rn+j}  - rs)$$
 is a periodization of $g$ with period $r$. 

Now set $v_l = \sum _{j=0}^{r-1} \delta _{l +jm} \in \bC ^{q}$
for  $l = 0, \dots , m-1$, where
$\delta _k (k) = 1 $ and $\delta _k(s) = 0$ for $s\neq k$. Clearly
these vectors are linearly independent. 
  Then
  \begin{align*}
    \big(P(x,0)v_l \big)_k &=  \sum _{j=0}^{r-1}    P(x,0)_{k,l
      + jm} \\
&= \sum _{j=0}^{r-1} \sum _{s\in \bZ
} g(x+\tfrac{l + jm }{m }   +\tfrac{k}{\beta} - rs) \\
&= \sum _{s\in \bZ
} g(x+\tfrac{l }{m}  +\tfrac{k}{\beta} - s) = 1 \, 
% &= \sum _{r\in \bZ
% } g(x+ \tfrac{\e}{2m+1} s +k/\beta - r) = 1\, .
  \end{align*}
for $k=0,\dots , p-1$ and $l = 0, \dots , m-1$. For the last equality
we have used hypothesis~\eqref{eq:2}. 

Setting $e= (1, \dots , 1)^T \in \bC ^p$, we have found $m$
linearly independent vectors $v_l $ such that $    P(x,0)v_l =
e$. Consequently the vectors $v_0 - v_l, l= 1, \dots , m-1$ are in the
kernel of $P(x,0)$. Since they  are also 
linearly independent, we know that 
$\mathrm{dim}\, (\mathrm{ker}\, P(x,0)) \geq m-1$. 
We obtain that 
\begin{align*}
\mathrm{rank}\, (P(x,0))& = rm  - \mathrm{dim}\, (\mathrm{ker}\,
P(x,0) \\
&\leq rm - (m-1) = (r-1)m +1 < rn+j = p  \, ,
 \end{align*}
by assumption on $m,n,r$. Thus the condition of Lemma~\ref{l1} is
violated and $\gab $ cannot be a frame. 
\end{proof}

\rems\ 1. Lemvig and Nielsen~\cite{LN15} observed that for the linear spline
$B_2 = \chi _{[-1/2,1/2]} \ast \chi _{[-1/2,1/2]}$ the lattices
$\tfrac{1}{2m+1} \bZ \times  \tfrac{n+1}{2} \bZ $ do not generate a
Gabor frame. With Proposition~\ref{p2} this is now a rigorous result. 

To put the observations of~\cite{LN15} into context, let
$$
\cF (g) = \{ (\alpha,\beta )\in \bR _+^2: \gab \,\, \text{ is a frame}
\}
$$
be the \emph{frame set} of $g$. This is  the set of all rectangular
lattices $\alpha \bZ \times \beta \bZ $ that generate a Gabor frame
with window $g$. Proposition~\ref{p2} says that the points
$(\tfrac{1}{m}, n+ \tfrac{j}{r})$ near the hyperbola $\alpha \beta =1$
do not belong to the frame set $\cF (g)$. 
 In addition, Lemvig and Nielsen observed that many of these points
 are not isolated points in the complement of  $\cF (g)$.  Their
 observations destroy the initial hope that the frame 
set   with respect to
$B$-spline windows  possesses a
simple structure. In fact, at this time it seems that the complexity
of the frame set of $B$-spline windows  resembles more that  of the characteristic
function $\chi _{[0,1]}$ determined in the stunning work of Dai and
Sun~\cite{sun13}. % I would not take any bets anymore on the nature of the
% frame set of $B$-splines. 

2. The partition-of-unity  condition~\eqref{eq:2} is a well known obstruction in Gabor analysis. It
was already observed by Del Prete~\cite{delprete} that $\gab $ fails
to be a frame when  $\beta = 2, 3, \dots
$ and $\alpha >0$ is arbitrary.  It comes as a surprise that
this condition excludes so many more lattices. In fact, for fixed
$n\in \bN , n\geq 2$ and $m=n+1$,  Proposition~\ref{p2} excludes  a countable
set of points of the
form $(\tfrac{1}{n+1}, n+\tfrac{j}{r}), r\in \bN$ with accumulation
point $(\tfrac{1}{n+1}, n+1)$ from $\cF (g)$. 

% It is clear that one can now play combinatorial games with this
% argument and find other obstructions, e.g., using $\alpha q = r \in
% \bN , r\geq 3$. I will not do that now, but wait for a reaction. 

\def\cprime{$'$} \def\cprime{$'$} \def\cprime{$'$} \def\cprime{$'$}
  \def\cprime{$'$} \def\cprime{$'$}

 \bibliographystyle{abbrv}
 \bibliography{general,new}

\end{document}